\documentclass[10pt,reqno]{amsart}
\usepackage[margin=1in]{geometry}

\usepackage{amsmath,amsthm,wasysym,bbold,tensor,mdframed,CJKutf8}
\usepackage[T1]{fontenc}
\usepackage{amssymb}
\usepackage{amsfonts}
\usepackage{mathscinet}
\usepackage[dvipsnames]{xcolor}
\usepackage{tikz-cd}
\usepackage{tikz,keyval,pgfsys,graphicx}
\usepackage{mathtools, enumerate, combelow}
\usepackage[shortlabels]{enumitem}

\usetikzlibrary{arrows, decorations.pathreplacing}
\usepgflibrary{arrows}
\usepackage{hyperref}

\newtheorem{thm}{\normalfont\scshape Theorem}[section]

\newtheorem{prop}[thm]{\normalfont\scshape Proposition}
\newtheorem{lem}[thm]{\normalfont\scshape Lemma}
\newtheorem{cor}[thm]{\normalfont\scshape Corollary}

\theoremstyle{definition}
\newtheorem{defn}[thm]{\normalfont\scshape Definition}

\theoremstyle{remark}
\newtheorem{rem}[thm]{Remark}
\theoremstyle{remark}

\DeclareMathOperator{\Ell}{Ell}

\allowdisplaybreaks

\begin{document}

\binoppenalty=10000
\relpenalty=10000

\numberwithin{equation}{section}

\newcommand{\uuu}{\mathfrak{u}}
\newcommand{\ccc}{\mathfrak{c}}
\newcommand{\qqq}{\mathfrak{q}}
\newcommand{\ddd}{\mathfrak{d}}
\newcommand{\QQ}{\mathbb{Q}}
\newcommand{\ZZ}{\mathbb{Z}}
\newcommand{\Hilb}{\mathrm{Hilb}}
\newcommand{\CC}{\mathbb{C}}
\newcommand{\PP}{\mathcal{P}}
\newcommand{\Hom}{\mathrm{Hom}}
\newcommand{\Rep}{\mathrm{Rep}}
\newcommand{\MM}{\mathfrak{M}}
\newcommand{\gl}{\mathfrak{gl}}
\newcommand{\ppp}{\mathfrak{p}}
\newcommand{\VV}{\mathcal{V}}
\newcommand{\NN}{\mathbb{N}}
\newcommand{\OO}{\mathcal{O}}
\newcommand{\HH}{\mathcal{H\kern-.44em H}}
\newcommand{\git}{/\kern-.35em/}
\newcommand{\KK}{\mathbb{K}}
\newcommand{\RR}{\mathbb{R}}
\newcommand{\Proj}{\mathrm{Proj}}
\newcommand{\quot}{\mathrm{quot}}
\newcommand{\core}{\mathrm{core}}
\newcommand{\Sym}{\mathrm{Sym}}
\newcommand{\FF}{\mathbb{F}}
\newcommand{\Span}[1]{\mathrm{span}\{#1\}}
\newcommand{\UTor}{U_{\qqq,\ddd}(\ddot{\mathfrak{sl}}_r)}
\newcommand{\Sss}{\mathcal{S}}
\newcommand{\TT}{\mathcal{T}}

\renewcommand{\ge}{\geqslant}
\renewcommand{\le}{\leqslant}

\title{Ext operators for wreath Macdonald polynomials}
\author{Seamus Albion Ferlinc}
\author{Joshua Jeishing Wen}
\address{Fakult\"{a}t f\"{u}r Mathematik, Universit\"{a}t Wien, Vienna, Austria}
\email{seamus.albion@univie.ac.at}
\email{joshua.jeishing.wen@univie.ac.at}


\begin{abstract}
We introduce a wreath Macdonald polynomial analogue of the Carlsson--Nekrasov--Okounkov vertex operator.
As an application, we prove a modular $(q,t)$-Nekrasov--Okounkov formula for $r\ge 3$ originally conjectured by Walsh and Warnaar.
\end{abstract}
\maketitle

\section{Introduction}

\subsection{Ext operators}\label{ExtIntro}
The term ``Ext operator'' comes from the geometry of moduli of sheaves on surfaces, which will not appear in the main body of the paper. 
We elaborate here on these geometric origins; the discussion extends to more general moduli of sheaves on more general surfaces, but we will confine ourselves to the Hilbert scheme of points on $\CC^2$.
Namely, for each $n$, we have the following identification of closed points:
\[
\Hilb_n(\CC^2)\cong\left\{ \mathcal{I}\subset\CC[x,y]\, |\, \dim \CC[x,y]\big/\mathcal{I}=n \right\}.
\]
Thus, $\Hilb_n(\CC^2)$ parameterizes certain ideal sheaves.
We set
\[
\Hilb\coloneq \bigsqcup_{n} \Hilb_n(\CC^2).
\]

After choosing coordinates, the two-dimensional torus $T$ acts naturally on $\CC^2$.
Given two equivariant sheaves $\mathcal{F}$ and $\mathcal{G}$ on $\CC^2$, one can define their $T$-equivariant $K$-theoretic Euler characteristic
\[
\chi_{\CC^2}(\mathcal{F},\mathcal{G})=\sum_{i}(-1)^i\mathrm{Ext}^i_{\CC^2}(\mathcal{F},\mathcal{G}).
\]
Carlsson, Nekrasov, and Okounkov \cite{COExt,CNO} consider the virtual class $E$ of vector bundles on $\Hilb\times\Hilb$ defined as follows: at the point corresponding to a pair of ideals $(\mathcal{I}_1,\mathcal{I}_2)\in\Hilb\times\Hilb$, the fiber of $E$ is
\[
E(\mathcal{I}_1,\mathcal{I}_2)=\chi_{\CC^2}(\mathcal{O}_{\CC^2},\mathcal{O}_{\CC^2})-\chi_{\CC_2}(\mathcal{I}_1,\mathcal{I}_2).
\]
The first summand $\chi_{\CC^2}(\mathcal{O}_{\CC^2},\mathcal{O}_{\CC^2})$ is just a normalization term; we call $E$ the \textit{Ext bundle}.

Lifting the $T$-action on $\CC^2$ naturally to $\Hilb$, the torus fixed points correspond to monomial ideals $\{\mathcal{I}_\lambda\}$, which can be indexed by partitions.
In \cite{COExt}, the authors compute
\begin{equation}
\sum_{k}(-u)^k {\textstyle\bigwedge^k} E(\mathcal{I}_\lambda,\mathcal{I}_\mu)=
\prod_{\square\in\lambda}
\left( 1-uq^{-a_\mu(\square)}t^{l_\lambda(\square)+1} \right)
\prod_{\square\in\mu}
\left( 1-uq^{a_\lambda(\square)+1}t^{-l_\mu(\square)} \right).
\label{ExtNek}
\end{equation}
Here, $a_\mu(\square)$ and $l_\mu(\square)$ are the (generalized) arm and leg lengths of the box; we define these quantities in Subsection~\ref{Young}.
The product on the right-hand-side of \eqref{ExtNek} is the \textit{Nekrasov factor}, a combinatorial quantity that has played a large role in the AGT correspondence \cite{AGT, AGTNotes}.

Either from Haiman's celebrated proof of the Macdonald positivity conjecture \cite{HaimanPoly} or from Grojnowski and Nakajima's Heisenberg action \cite{GrojHilb, NakHilbAnnals}, it is well-known that one should study cohomology or $K$-theory of $\Hilb$ in terms of the ring of symmetric functions $\Lambda$ \cite{Mac}.
In $K$-theory, the class of the skyscraper sheaf $[\mathcal{O}_{\mathcal{I}_\lambda}]$ corresponds to the \textit{modified Macdonald polynomial} $H_\lambda$.
One of the key results from \cite{COExt, CNO} is a description of the left-hand-side of (\ref{ExtNek}) in terms of natural operators on symmetric functions.
Thus, they are able to reproduce (\ref{ExtNek}) using certain operators applied to $H_\lambda$ and $H_\mu$.

\subsection{Wreath Macdonald polynomials}
This paper is concerned with a generalization of \cite{CNO} where one replaces the Macdonald polynomials with \textit{wreath Macdonald polynomials}.
Recall that under the Frobenius characteristic, we can use $\Lambda$ to describe the representation rings of symmetric groups $\Sigma_n$:
\[
\Lambda\cong\bigoplus_{n}\mathrm{Rep}(\Sigma_n).
\]
Defined by Haiman \cite{Haiman}, the wreath Macdonald polynomials are characters for the wreath product $\ZZ/r\ZZ\wr\Sigma_n$.
Fixing $r$ and letting $n$ vary, we also have a wreath Frobenius characteristic \cite[Appedix I.B]{Mac}:
\[
\Lambda^{\otimes r}\cong\bigoplus_{n}\mathrm{Rep}(\ZZ/r\ZZ\wr\Sigma_n).
\]
The wreath Macdonald polynomials are also indexed by (single) partitions; we will denote them again by $\{H_\lambda\}$.
The decomposition of a partition into its $r$-core and $r$-quotient (cf.~Subsection~\ref{CoreQuot} below) plays a role in their definition, and letting $\lambda$ range over all of those with the same $r$-core, we obtain a basis of $\CC(q,t)\otimes \Lambda^{\otimes r}$.

The geometric story told in Subsection~\ref{ExtIntro} above also has a wreath analogue.
By work of Gordon \cite{Gordon} and Bezrukavnikov and Finkelberg \cite{BezFink} (cf.~\cite{Haiman} as well), the wreath Macdonald polynomials are also intimately tied to the $K$-theory of $\Hilb$.
Namely, consider the action of $\ZZ/r\ZZ$ on $\CC^2$ via matrices
\[
\begin{pmatrix}
e^{\frac{k\pi i}{r}} & 0\\
0 & e^{-\frac{k\pi i}{r}}
\end{pmatrix}.
\]
Lifting this action to $\Hilb$, one can consider the fixed subvarieties $\Hilb^{\ZZ/r\ZZ}$.
These are still smooth, and they still contain the monomial ideals $\{\mathcal{I}_\lambda\}$.
The aforementioned work assign the class of the skyscraper sheaf $\left[ \mathcal{O}_{\mathcal{I}_\lambda} \right]$, now in $K_T( \Hilb^{\ZZ/r\ZZ})$, to $H_\lambda$.

Restricted to $\Hilb^{\ZZ/r\ZZ}\times\Hilb^{\ZZ/r\ZZ}$, the Ext bundle $E$ becomes a bundle of $\ZZ/r\ZZ$-modules, and we can take its invariants $E^{\ZZ/r\ZZ}$. 
Equation (\ref{ExtNek}) becomes
\[
\sum_{k}(-u)^k {\textstyle\bigwedge^k} \bigg(E(\mathcal{I}_\lambda,\mathcal{I}_\mu)^{\ZZ/r\ZZ}\bigg)= N_{\lambda,\mu}(u),
\]
where 
\begin{equation}\label{Eq_NF}
N_{\lambda,\mu}(u)\coloneq\prod_{\substack{\square\in\lambda\\ h_{\mu,\lambda}(\square)\equiv 0\:\mathrm{mod}\: r}}
\left( 1-uq^{-a_\mu(\square)}t^{l_\lambda(\square)+1} \right)
\prod_{\substack{\square\in\mu\\ h_{\lambda,\mu}(\square)\equiv 0\:\mathrm{mod}\: r}}
\left( 1-uq^{a_\lambda(\square)+1}t^{-l_\mu(\square)} \right).
\end{equation}
Here, the quantity $h_{\mu,\lambda}(\square)$ is the mixed hook length, defined in Subsection~\ref{Young}.
The expression
$N_{\lambda,\mu}(u)$ is the Nekrasov factor for ALE space, considered in \cite[Equation~(5.33)]{AKMMSZ}.

Our main result is the following.

\begin{thm}\label{MainThm}
Let $r\ge 3$.
Define
\[
\mathsf{W}(u)
\coloneq\Omega\left[ \frac{(1-u^{-1})X^{(0)}}{(1-q\sigma^{-1})(t\sigma-1)}\right]
\TT\left[ (1-uqt)X^{(0)} \right].
\]
For $\lambda$ and $\mu$ with the same $r$-core, we have
\[
\left\langle H_\mu^\dagger, \mathsf{W}(u)H_\lambda\right\rangle_{q,t}'=u^{-|\quot(\mu)|}N_{\lambda,\mu}(u).
\]
\end{thm}
\noindent All the notation in the theorem is introduced in Sections \ref{Partitions} and \ref{WMT}.
Our formula should also hold for $r=2$.
The only obstruction is that the results from \cite{WreathEigen} involving the quantum toroidal algebra are only proved for $r\ge 3$; the algebra for $r=2$ has different formulas.

Our method of proof is philosophically identical to \cite{CNO}.
Although not obvious at first, it turns out that Theorem \ref{MainThm} is intimately tied to the reciprocity property for Macdonald polynomials.
In \cite{CNO}, this is manifested in their use of the \textit{Cherednik--Macdonald--Mehta identity} \cite{ChereMM}.
Here, we phrase this in terms of \textit{Tesler's identity} \cite{GarsiaTesler, GHT}, which was proven for wreath Macdonald polynomials by Romero and the second author \cite{RW}.

\subsection{Nekrasov--Okounkov formulas}
The ordinary Nekrasov--Okounkov formula is the following expansion for an arbitrary complex power of the Dedekind eta function:
\begin{equation}\label{Eq_NO}
\sum_{\lambda}T^{|\lambda|}\prod_{\square\in\lambda}
\bigg(1-\frac{z}{h_\lambda^2(\square)}\bigg)
=\prod_{i\ge 1}(1-T^i)^{z-1}.
\end{equation} 
Note that $h_\lambda(\square)=h_{\lambda,\lambda}(\square)$ is the ordinary hook length of the box $\square$.
One may either take $0<T<1$ and $z\in\CC$ or $T\in\CC$ with $|T|<1$ and $z\in\RR$.
This was discovered by Nekrasov and Okounkov \cite[Equation~(6.12)]{NO} in the course of their proof of Nekrasov's conjecture \cite{Nekrasov04}. 
It was found independently by Westbury as a hook-length formula for the D'Arcais polynomials \cite{Westbury06}.

Generalizations of the Nekrasov--Okounkov formula abound in the literature;
see, for instance, \cite{CRV,DH11,HanNO,HJ11,Petreolle16,Petreolle16II,RainsWar16,Wahiche22,WalshWar} and references therein.
One such generalization is Han's modular version wherein the product in the summand is replaced by a product over boxes with hook length congruent to zero modulo $r$ for a positive integer $r$ \cite[Theorem~1.4]{HanNO}.
This was further expanded on by Han and Ji by way of their so-called ``multiplication-addition theorem'' \cite{HJ11} which allows for modular analogues of a vast swathe of identities involving hook-lengths.

In a different direction, Rains and Warnaar \cite[Theorem~1.3]{RainsWar16} and Carlsson and Rodriguez Villegas \cite[Theorem~1.0.2]{CRV} independently proved a $(q,t)$-deformation of the Nekrasov--Okounkov formula. 
To state this, let
\[
(a_1,\ldots, a_n; q_1,\ldots, q_m)_\infty
=
\prod_{i=1}^n\prod_{j=1}^m\prod_{k=0}^\infty \big( 1-a_iq_j^k \big)
\]
denote the infinite multiple $q$-Pochhammer symbol.
Then 
\begin{equation}\label{Eq_qtNO}
\sum_{\lambda}T^{|\lambda|}\prod_{\square\in\lambda}
\frac{(1-uq^{a_\lambda(\square)+1}t^{l_\lambda(\square)})
(1-u^{-1}q^{a_\lambda(\square)}t^{l_\lambda(\square)+1})}
{(1-q^{a_\lambda(\square)+1}t^{l_\lambda(\square)})
(1-q^{a_\lambda(\square)}t^{l_\lambda(\square)+1})}
=\frac{(uqT,u^{-1}tT;q,t,T)_\infty}{(T,tT;q,t,T)_\infty}.
\end{equation}
For $q=t$ this reduces to a $q$-analogue of the Nekrasov--Okounkov formula which may be found in \cite{DH11} or \cite{INRS12}.
Further setting $u=q^z$ and taking the limit $q\to 1$ then recovers 
\eqref{Eq_NO}.
Alternatively, by specializing $u=(t/q)^{1/2}$ and then replacing $(q,t)\mapsto(q^{-2},t^2)$ one obtains an identity conjectured by Hausel and Rodriguez Villegas \cite[Conjecture~4.3.2]{HRV} (our $(q,t)$ is their $(w,z)$).
This in turn is equivalent to the genus-one case of their much more general conjecture \cite[Conjecture~4.2.1]{HRV}. 
This remarkable conjecture gives an explicit expression for the mixed Hodge polynomial of the twisted character variety of a closed Riemann surface of genus $g$ in terms of a plethystic logarithm involving a generalized hook-product.

The $g=1$ case of this general conjecture is what motivated Rains and Warnaar and Carlsson and Rodriguez Villegas to prove the $(q,t)$-Nekrasov--Okounkov formula \eqref{Eq_qtNO}.
Their proofs are built in different ideas.
Rains and Warnaar's approach is based on evaluating a sum of specialized ordinary (i.e., non-modified) skew Macdonald polynomials in two different ways.
On the other hand, Carlsson and Rodriguez Villegas show that \eqref{Eq_qtNO} may be obtained by taking the trace of the Carlsson--Nekrasov--Okounkov vertex operator of which our $\mathsf{W}(u)$ is the wreath analogue.
In particular, the modified Macdonald polynomial case of Theorem~\ref{MainThm}, stated as Theorem~3.0.1 of \cite{CRV}, is key in their proof.

By following Carlsson and Rodriguez Villegas and taking the trace of $\mathsf{W}(u)$ we obtain the following modular refinement of \eqref{Eq_qtNO} which was already conjectured by Walsh and Warnaar \cite[Conjecture~8.1]{WalshWar}.
\begin{thm}\label{Thm_WW}
Let $r\ge 3$ be an integer. Then for any $r$-core $\alpha$, we have
\begin{multline}\label{WW1}
\sum_{\substack{\lambda\\ \core(\lambda)=\alpha}}T^{|\quot(\lambda)|}
\prod_{\substack{\square\in\lambda\\ h_\lambda(\square)\equiv 0\:\mathrm{mod}\: r}}
\frac{(1-uq^{a_\lambda(\square)+1}t^{l_\lambda(\square)})(1-u^{-1}q^{a_\lambda(\square)}t^{l_\lambda(\square)+1})}{(1-q^{a_\lambda(\square)+1}t^{l_\lambda(\square)})(1-q^{a_\lambda(\square)}t^{l_\lambda(\square)+1})} \\
=\frac{1}{(T;T)_\infty^r}\prod_{i=1}^r\frac{(uq^it^{r-i}T,u^{-1}q^{r-i}t^iT;q^r,t^r,T)_\infty}{(q^it^{r-i}T,q^{r-i}t^iT;q^r,t^r,T)_\infty}.
\end{multline}
\end{thm}
Again, as for Theorem~\ref{MainThm}, we expect this to hold also in the case $r=2$.
Replacing $T\mapsto ST^r$, scaling both sides by $T^{|\alpha|}$ and summing over all $r$-cores using the generating function \cite[p.~13]{Mac}
\[
\sum_\alpha T^{|\alpha|}=\frac{(T^r;T^r)_\infty^r}{(T;T)_\infty},
\]
we obtain a weaker form of the theorem stated in \cite[Conjecture~1.2]{WalshWar}:
\begin{multline*}
\sum_{\lambda}S^{|\quot(\lambda)|}T^{|\lambda|}
\prod_{\substack{\square\in\lambda\\ h_\lambda(\square)\equiv 0\:\mathrm{mod}\: r}}
\frac{(1-uq^{a_\lambda(\square)+1}t^{l_\lambda(\square)})(1-u^{-1}q^{a_\lambda(\square)}t^{l_\lambda(\square)+1})}{(1-q^{a_\lambda(\square)+1}t^{l_\lambda(\square)})(1-q^{a_\lambda(\square)}t^{l_\lambda(\square)+1})} \\
=\frac{(T^r;T^r)_\infty^r}{(T;T)_\infty(ST^r;ST^r)_\infty^r}\prod_{i=1}^r\frac{(uq^it^{r-i}ST^r,u^{-1}q^{r-i}t^iST^r;q^r,t^r,ST^r)_\infty}{(q^it^{r-i}ST^r,q^{r-i}t^iST^r;q^r,t^r,ST^r)_\infty}.
\end{multline*}
Walsh and Warnaar conjectured \eqref{WW1} based on the $q=0$ and $t=0$ cases,
combinatorial identities involving products over boxes with arm- and leg-length
zero, respectively.
These special cases are established purely combinatorially using a variant of the core-quotient construction and an analogue of the ``multiplication theorem'' of Han and Ji \cite{HJ11}.
However, it does not appear that there is such a nice combinatorial explanation for these modular variants at the $(q,t)$-level.

\subsection{Outline of the paper}
The paper reads as follows.
In the next section we cover the basics of partitions, the core-quotient construction and plethystic exponentials. Section~\ref{WMT} is then devoted to wreath Macdonald polynomial theory, including the key results from \cite{RW} which we require for our main construction.
Finally, in Section~\ref{ExtOp}, we introduce the Ext operator, prove both Theorems~\ref{MainThm} and \ref{Thm_WW}, and comment on an elliptic analogue of the latter stemming from the work of Walsh and Warnaar.

\section{Partitions}\label{Partitions}
Here we cover the basics of partitions and the core-quotient construction,
following closely \cite{WreathEigen}; see also \cite[p.~12--15]{Mac}.

\subsection{Basic notation}

A partition $\lambda=(\lambda_1,\lambda_2,\lambda_3,\ldots)$ is a nonincreasing sequence of nonnegative integers with only finitely many nonzero entries:
\[
\lambda_1\ge \lambda_2\ge\lambda_3\ge\cdots\ge 0.
\]
The sum of the entries is well-defined; we denote it by $|\lambda|=\sum_{i\ge1}\lambda_i$.
The number of nonzero $\lambda_i$ is called the length and is denoted by
$\ell(\lambda)$.
Later on, given a positive integer $r$, we will also work with $r$-tuples of partitions $\vec{\lambda}=(\lambda^{(0)},\ldots, \lambda^{(r-1)})$.
Entries in the tuple will be indexed using superscripts, while entries of an individual partition will be indexed by subscripts.
We extend the notation $|\vec{\lambda}|=\sum_i|\lambda^{(i)}|$.
Finally, we use $\ge$ to denote the dominance order on partitions of a fixed integer $n$: $\lambda\ge\mu$ if for all $k\ge 1$,
\[
\lambda_1+\cdots+\lambda_k\ge\mu_1+\cdots+\mu_k.
\]

\subsection{Diagrams}
Here, we review two ways to visualize a partition: through its \textit{Young diagram} and its \textit{Maya diagram}.

\subsubsection{Young diagrams}\label{Young}
For a partition $\lambda$, its Young diagram is the following set of lattice points:
\[
\{ (a,b)\in\ZZ^2 \mid 0\le a\le\lambda_b-1 \}.
\]
To each lattice point, we assign a box; we will draw the resulting arrangement of boxes following the French convention.
For example, the Young diagram of $\lambda=(6, 4, 1)$ is shown in Figure \ref{fig:young} below.
Note that the bottom left corner has coordinate $(0,0)$.
We will conflate a partition with its Young diagram and make statements like $(a,b)\in\lambda$.
\begin{figure}[ht]
\centering
\begin{tikzpicture}[scale=0.5]
\foreach \i [count=\ii] in {6,4,1}
\foreach \j in {1,...,\i}{\draw (\j,1+\ii) rectangle (\j+1,\ii);}
\end{tikzpicture}
\caption{The Young diagram of the partition $\lambda=(6,4,1)$.}
\label{fig:young}
\end{figure}
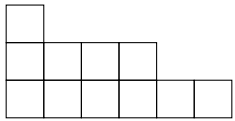

This visual representation clarifies some of the basic definitions concerning partitions.
The \textit{transpose} or \textit{conjugate} partition ${}^t\lambda=({}^t\lambda_1,{}^t\lambda_2,\ldots)$ is defined to be
\[
{}^t\lambda_i=\#\{1\le j\le\ell(\lambda) \mid \lambda_j\ge i\}.
\]
Its Young diagram is just the one for $\lambda$ reflected across the diagonal line $y=x$.
For a box $\square=(a,b)\in\ZZ_{\ge 0}^2$ (not necessarily in $\lambda$), we define its \emph{arm-} and \emph{leg-length} by
\[
a_\lambda(\square)\coloneq\lambda_{b+1}-a-1 \quad\text{and}\quad
l_\lambda(\square)\coloneq{}^t\lambda_{a+1}-b-1.
\]
When $\square\in\lambda$, these quantities are easily visualized in the Young diagram: they count the number of boxes to the right of and above $\square$, respectively.
For a pair of partitions $\lambda,\mu$ and a box $\square$ lying in one of them we define the \emph{mixed hook-length} by
\[
h_{\lambda,\mu}(\square)\coloneq a_\lambda(\square)+l_\mu(\square)+1.
\]
The ordinary hook-length of a box $\square\in\lambda$ is then $h_{\lambda,\lambda}(\square)=h_{\lambda}(\square)$.
Finally, the \textit{content} of a box $\square=(a,b)$ is $c_\square:=b-a$.
This is just the SW-to-NE diagonal that $\square$ lies on.

\subsubsection{Maya diagrams}
A Maya diagram is a function $m:\ZZ\longrightarrow\{\pm 1\}$ such that
\[
m(n)=
\begin{cases}
-1 & n\gg0,\\
1 & n\ll 0.
\end{cases}
\]
We associate to $m$ a visual representation using beads.
Namely, consider a string of black and white beads indexed by $\ZZ$ where the bead at position $n$ is black if $m(n)=1$ and white if $m(n)=-1$.
These beads will be arranged horizontally with the index increasing towards the \textit{left}.
Between indices $0$ and $-1$, we draw a notch and call it the \textit{central line}.

Left of the central line (nonnegative values of $n$) all but finitely many beads will be white.
Likewise, all but finitely many beads right of the central line (negative values of $n$) will be black.
The \textit{charge} $c(m)$ is the difference in the number of discrepancies:
\[
c(m)=\#\{ n\ge 0 \mid m(n)=1 \}-\#\{n<0 \mid m(n)= -1\}.
\]
The \textit{vacuum diagram} is the Maya diagram with only white beads left of the central line and only white beads right of the central line.

\subsubsection{Young--Maya correspondence}
To a partition $\lambda$, we associate a Maya diagram $m_\lambda$ via its \textit{edge sequence}.
Namely, tilt $\lambda$ by $45$ degrees counterclockwise and draw the level sets for the content.
We label by $n$ the gap between content lines $n$ and $n+1$.
Within the gap labeled by $n$, the outer edge of $\lambda$ either has slope $1$ or $-1$; $m_\lambda(n)$ takes the value of the slope of that segment.
Outside of $\lambda$, we assign the default limit values of white beads left of the central line and black beads right of the central line.
For example, Figure \ref{fig:maya-diagram} shows $m_\lambda$ correspondence for $\lambda=(6,4,1)$.

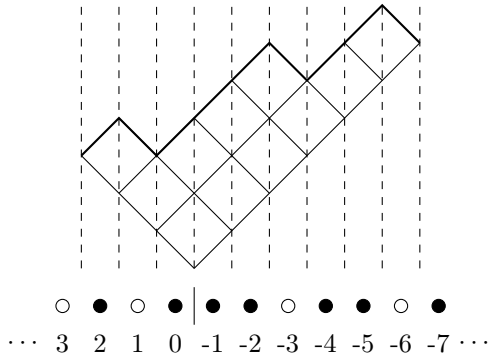
\begin{figure}[ht]
\begin{tikzpicture}[scale=.5]
\draw[thick] (-2.5,5)--(-1.5,6)--(-0.5,5)--(2.5,8)--(3.5,7)--(5.5,9)--(6.5,8);
\draw (-1.5,4)--(0.5,6);
\draw (-.5,3)--(3.5,7);
\draw (.5,2)--(6.5,8);
\draw (-2.5,5)--(.5,2);
\draw (-1.5,6)--(1.5,3);
\draw (0.5,6)--(2.5,4);
\draw (1.5,7)--(3.5,5);
\draw (3.5,7)--(4.5,6);
\draw (4.5,8)--(5.5,7);
\draw (-4,0) node {$\cdots$};
\draw (-3,0) node {3};
\draw (-3,1) circle (5pt);
\draw (-2,0) node {2};
\draw[fill=black] (-2,1) circle (5pt);
\draw (-1,0) node {1};
\draw (-1,1) circle (5pt);
\draw (0,0) node {0};
\draw[fill=black] (0,1) circle (5pt);
\draw (.5,.5)--(.5,1.5);
\draw (1,0) node {-1};
\draw[fill=black] (1,1) circle (5pt);
\draw (2,0) node {-2};
\draw[fill=black] (2,1) circle (5pt);
\draw (3,0) node {-3};
\draw (3,1) circle (5pt);
\draw (4,0) node {-4};
\draw[fill=black] (4,1) circle (5pt);
\draw (5,0) node {-5};
\draw[fill=black] (5,1) circle (5pt);
\draw (6,0) node {-6};
\draw (6,1) circle (5pt);
\draw (7,0) node {-7};
\draw[fill=black] (7,1) circle (5pt);
\draw (8,0) node {$\cdots$};
\draw[dashed] (1.5,2)--(1.5,9);
\draw[dashed] (2.5,2)--(2.5,9);
\draw[dashed] (3.5,2)--(3.5,9);
\draw[dashed] (4.5,2)--(4.5,9);
\draw[dashed] (5.5,2)--(5.5,9);
\draw[dashed] (6.5,2)--(6.5,9);
\draw[dashed] (0.5,2)--(0.5,9);
\draw[dashed] (-0.5,2)--(-0.5,9);
\draw[dashed] (-1.5,2)--(-1.5,9);
\draw[dashed] (-2.5,2)--(-2.5,9);
\end{tikzpicture}
 \label{fig:maya-diagram}
\caption{The Young--Maya correspondence for the partition $\lambda=(6,4,1)$.}
\end{figure}

\begin{prop}\label{YoungMaya}
The map $\lambda\mapsto m_\lambda$ is a bijection between partitions and Maya diagrams of charge $0$.
\end{prop}
Under this bijection the empty partition corresponds to the vacuum diagram.

\subsection{Core-quotient decomposition}\label{CoreQuot}
For now on, we fix an integer $r\ge 1$.
Key to the definition of wreath Macdonald polynomials is the decomposition of a partition to its $r$-core and $r$-quotient.
A \textit{ribbon} of $\lambda$ is a contiguous subset of $\lambda$ satisfying the following:
\begin{itemize}
\item it does not contain a $2\times 2$ square of boxes;
\item its removal results in another partition.
\end{itemize}
An $r$-core partition is a partition that does not contain a ribbon of length $r$.
Given a partition $\lambda$, its $r$-core, denoted $\core(\lambda)$, is the resulting partition when one removes ribbons of length $r$ until one can no longer do so.
The construction below shows that this operation is well-defined in that it results in a unique partition $\core(\lambda)$.

\subsubsection{The $r$-quotient}
First, we consider the \textit{quotient subdiagrams} of $m_\lambda$.
Namely, for $0\le i\le r-1$, we define
\[
m_\lambda^{(i)}(n)\coloneq m_\lambda(i+nr).
\]
The subdiagram $m_\lambda^{(i)}$ will have a charge $c_i$ that can be nonzero.
Shifting the central line of $m_\lambda^{(i)}$ to the left by $c_i$ (right if $c_i<0$), we obtain a charge $0$ Maya diagram.
Applying Proposition \ref{YoungMaya}, we obtain a partition $\lambda^{(i)}$.
The $r$-quotient is the tuple
\[
\quot(\lambda)\coloneq \left( \lambda^{(0)},\lambda^{(1)},\ldots,\lambda^{(r-1)} \right)
\] 
Note that we do not emphasize $r$ in the notation $\quot(\lambda)$ because $r$ is fixed throughout the paper.
It is useful to view the entries in $\quot(\lambda)$ as indexed by $\ZZ/r\ZZ$.
We compute the $\quot(\lambda)$ for $\lambda=(6,4,1)$ and $r=3$ in Figure \ref{fig:quotient}.

\begin{figure}[ht]
\centerline{
\begin{tikzpicture}[scale=.5]
\draw (-8,.5) node {$m_\lambda$:};
\draw (-6,0) node {$\cdots$};
\draw (-5,0) node {3};
\draw (-5,1) circle (5pt);
\draw (-4,0) node {2};
\draw[fill=black] (-4,1) circle (5pt);
\draw (-3,0) node {1};
\draw (-3,1) circle (5pt);
\draw (-2,0) node {0};
\draw[fill=black] (-2,1) circle (5pt);
\draw (-1.5,.5)--(-1.5,1.5);
\draw (-1,0) node {-1};
\draw[fill=black] (-1,1) circle (5pt);
\draw (0,0) node {-2};
\draw[fill=black] (0,1) circle (5pt);
\draw (1,0) node {-3};
\draw (1,1) circle (5pt);
\draw (2,0) node {-4};
\draw[fill=black] (2,1) circle (5pt);
\draw (3,0) node {-5};
\draw[fill=black] (3,1) circle (5pt);
\draw (4,0) node {-6};
\draw (4,1) circle (5pt);
\draw (5,0) node {-7};
\draw[fill=black] (5,1) circle (5pt);
\draw (6,0) node {$\cdots$};
\draw (-15,-5.5) node {$m^{(0)}_\lambda:$};
\draw (-13,-6) node {$\cdots$};
\draw (-12,-6) node {3};
\draw (-12,-5) circle (5pt);
\draw (-11,-6) node {0};
\draw[fill=black] (-11,-5) circle (5pt);
\draw (-10,-6) node {-3};
\draw (-10,-5) circle (5pt);
\draw (-9.5,-4)--(-11.5,-2)--(-10.5,-1)--(-8.5,-3)--(-9.5,-4);
\draw (-10.5,-3)--(-9.5,-2);
\draw[dashed] (-9.5, -5.5)--(-9.5,-4.5);
\draw (-9,-6) node {-6};
\draw (-9,-5) circle (5pt);
\draw (-8,-6) node {-9};
\draw[fill=black] (-8,-5) circle (5pt);
\draw (-7,-6) node {$\cdots$};
\draw (-10.5,-5.5)--(-10.5,-4.5);
\draw (-4,-5.5) node {$m^{(1)}_\lambda:$};
\draw (-2,-6) node {$\cdots$};
\draw (-1,-6) node {1};
\draw (-1,-5) circle (5pt);
\draw (-0.5,-5.5)--(-0.5,-4.5);
\draw (-0.5, -3) node {$\varnothing$};
\draw (0,-6) node {-2};
\draw[fill=black] (0,-5) circle (5pt);
\draw (1,-6) node {$\cdots$};
\draw (4, -5.5) node {$m^{(2)}_\lambda:$};
\draw (6,-6) node {$\cdots$};
\draw (7,-6) node {5};
\draw (7,-5) circle (5pt);
\draw[dashed] (7.5,-5.5)--(7.5,-4.5);
\draw (7.5,-3) node {$\varnothing$};
\draw (8,-6) node {2};
\draw[fill=black] (8,-5) circle (5pt);
\draw (8.5,-5.5)--(8.5,-4.5);
\draw (9,-6) node {$\cdots$};
\end{tikzpicture}
}
\caption{The $r$-quotient for $\lambda = (6,4,1)$ and $r=3$. The original central line is drawn with a solid line, and the central line shifted left by $c_i$ is drawn as a dashed line.}
\label{fig:quotient}
\end{figure}
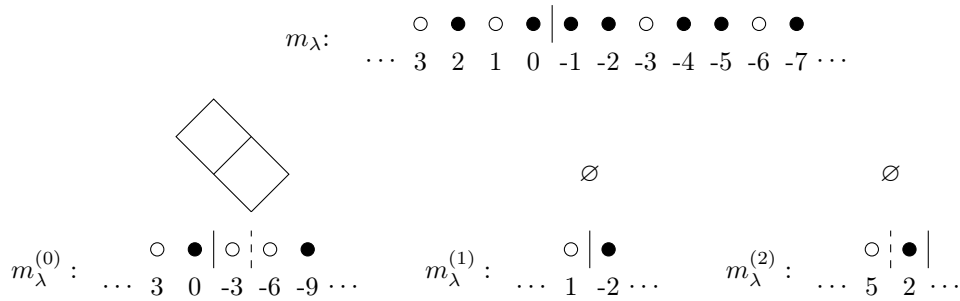

The following is an elementary exercise.
\begin{prop}\label{QuotHook}
We have the following equality:
\[
|\quot(\lambda)|=\#\left\{ \square\in\lambda\,\middle|\, h_\lambda(\square)\equiv 0\:\mathrm{mod}\: r \right\}.
\]
\end{prop}

\subsubsection{The $r$-core}
For the $r$-core, we replace each $m_\lambda^{(i)}$ with a vacuum diagram with central line shifted left by the charge $c_i$.
Reconstituting these quotient diagrams back into one big Maya diagram, we obtain a Maya diagram of charge zero.
Its associated partition is $\core(\lambda)$.
We illustrate this for $\lambda=(6,4,1)$ and $r=3$ in Figure \ref{fig:core}.

\begin{figure}[ht]
\centering
\begin{tikzpicture}[scale=.5]
\draw (-6,.5) node {$m_\lambda$:};
\draw (-4,0) node {$\cdots$};
\draw (-3,0) node {3};
\draw (-3,1) circle (5pt);
\draw (-2,0) node {2};
\draw[fill=black] (-2,1) circle (5pt);
\draw (-1,0) node {1};
\draw (-1,1) circle (5pt);
\draw (0,0) node {0};
\draw[fill=black] (0,1) circle (5pt);
\draw (.5,.5)--(.5,1.5);
\draw (1,0) node {-1};
\draw[fill=black] (1,1) circle (5pt);
\draw (2,0) node {-2};
\draw[fill=black] (2,1) circle (5pt);
\draw (3,0) node {-3};
\draw (3,1) circle (5pt);
\draw (4,0) node {-4};
\draw[fill=black] (4,1) circle (5pt);
\draw (5,0) node {-5};
\draw[fill=black] (5,1) circle (5pt);
\draw (6,0) node {-6};
\draw (6,1) circle (5pt);
\draw (7,0) node {-7};
\draw[fill=black] (7,1) circle (5pt);
\draw (8,0) node {$\cdots$};
\draw (0.5,-2) node {$\downarrow$};
\draw (-4,-9) node {$\cdots$};
\draw (-3,-9) node {3};
\draw (-3,-8) circle (5pt);
\draw (-2,-9) node {2};
\draw[fill=black] (-2,-8) circle (5pt);
\draw (-1,-9) node {1};
\draw (-1,-8) circle (5pt);
\draw (0,-9) node {0};
\draw (0,-8) circle (5pt);
\draw (.5,-8.5)--(.5,-7.5);
\draw (1,-9) node {-1};
\draw[fill=black] (1,-8) circle (5pt);
\draw (2,-9) node {-2};
\draw[fill=black] (2,-8) circle (5pt);
\draw (3,-9) node {-3};
\draw (3,-8) circle (5pt);
\draw (4,-9) node {-4};
\draw[fill=black] (4,-8) circle (5pt);
\draw (5,-9) node {$\cdots$};
\draw (.5,-7)--(3.5, -4)--(2.5,-3)--(0.5, -5)--(-1.5,-3)--(-2.5,-4)--(.5,-7);
\draw (-1.5,-5)--(-0.5,-4);
\draw (-.5,-6)--(.5,-5);
\draw (0.5,-5)--(1.5,-6);
\draw (1.5,-4)--(2.5,-5);
\end{tikzpicture}\caption{The 3-core of $\lambda=(6,4,1)$.}
\label{fig:core}
\end{figure}
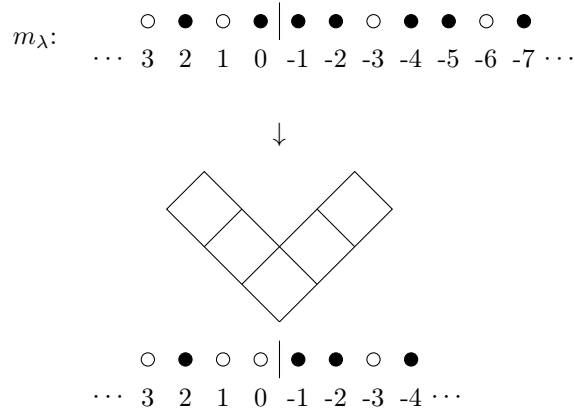

In a sense, $\quot(\lambda)$ records a box for each ribbon of length $r$ that has been removed and $\core(\lambda)$ records what is left over.
To summarize, we have the following proposition.
\begin{prop}
The map
\[
\lambda\longmapsto(\core(\lambda),\quot(\lambda))
\]
yields a bijection
\[
\{\hbox{partitions}\}\longleftrightarrow\{\hbox{$r$-cores}\}\times\{\hbox{$r$-tuples of partitions}\}.
\]
\end{prop}

\subsection{Characters}
Now we introduce our Macdonald parameters $(q,t)$.
To a box $\square=(a,b)\in\ZZ^2$, we set its character to be
\[
\chi_\square\coloneq q^at^b.
\]
The \textit{color} of $\square=(a,b)$ will be the class of $b-a$ modulo $r$; we denote this by $\bar{c}_{\square}$.

\subsubsection{Sum conventions}\label{CharSum}
For a (possibly infinite) set $A\subset\ZZ^2$, consider a sum of characters
\[
S=\sum_{\square\in A}n_\square\chi_\square
\]
for some integers $\{n_\square\}_{\square\in A}$.
We will subdivide $S$ according to color: for $i\in\ZZ/r\ZZ$,
\[
S^{(i)}\coloneq\sum_{\substack{\square\in A\\ \bar{c}_\square=i}}n_\square\chi_\square.
\]
We will occasionally denote $S^\bullet=S$ to make the reader conscious of this subdivision.
Notice that multiplication by $q$ cyclically lowers the color parameter whereas multiplication by $t$ does the opposite:
\begin{equation*}
\begin{aligned}
\left( qS \right)^{(i-1)}&= q\left( S^{(i)} \right),&
\left( tS \right)^{(i+1)}&= t\left( S^{(i)} \right).
\end{aligned}
\end{equation*}

To a partition $\lambda$, we will often work with the following two sums:
\begin{align*}
B_\lambda^\bullet&\coloneq \sum_{\square\in\lambda}\chi_\square\\
D_\lambda^\bullet&\coloneq (1-q)(1-t)B_\lambda^\bullet-1.
\end{align*}
Note that if we set
\begin{align*}
A_i(\lambda)&\coloneq\, \hbox{the addable boxes of color $i$}\\
R_i(\lambda)&\coloneq\, \hbox{the removable boxes of color $i$}
\end{align*}
then
\[
D_\lambda^{(i)}=qt\sum_{\square\in R_(\lambda)}\chi_\square-\sum_{\square\in A_i(\lambda)}\chi_\square.
\]

Finally, let us end on two points.
First, we will encounter rational functions in $(q,t)$ with denominators of the form $(1-q^kt^l)$.
We will by default expand this into series assuming $|q|,|t|<1$.
Second, we define a character inversion operation by
\[
\overline{S}\coloneq S\,\bigg|_{(q,t)\mapsto(q^{-1},t^{-1})}.
\]
To avoid convergence issues, we will only apply this to finite sums.
Note that following our color conventions,
\[
\overline{S}^{(i)}= \overline{\left( S^{(-i)} \right)}.
\]
This is contrary to the $(-)_*$ operation on characters defined in \cite{RW}, where
\[
S_*^{(i)}= S^{(i)}\,\bigg|_{(q,t)\mapsto (q^{-1},t^{-1})}.
\]
The reader should keep this in mind when we use formulas from \cite{RW} involving inverse characters.

\subsubsection{Plethystic exponential}
For a sum of characters 
\[
S=\sum_{\square\in A}n_\square\chi_\square,
\]
with positive exponents, we define its plethystic exponential to be
\[
\Omega[S]:=\exp\left( \sum_{k>0}\sum_{\square\in A} n_\square\chi_\square^k \right)=\prod_{\square\in A}\left( 1-\chi_\square \right)^{-n_\square}.
\]
Note that our assumption $|q|,|t|<1$ is important here.
To handle characters with arbitrary exponents, we will introduce a color $0$ parameter $u$ such that $|u|\ll 1$.

\subsubsection{Nekrasov factors}\label{NekFac}
Recall the Nekrasov factor \eqref{Eq_NF} from the introduction.
The following lemma gives an expression for $N_{\lambda,\mu}(u)$ in terms of a plethystic exponential.

\begin{lem}\label{NekExp}
The Nekrasov factor $N_{\lambda,\mu}(u)$ can be written as
\[
N_{\lambda,\mu}(u)=\Omega\left[\left(\frac{uqtD_\lambda\overline{D}_\mu}{(1-q)(1-t)}\right)^{(0)}-\left(\frac{uqt}{(1-q)(1-t)}\right)^{(0)} \right].
\]
\end{lem}

\begin{proof}
Set
\[
E_{\lambda,\mu}^\bullet\coloneq \sum_{\square\in\lambda}
q^{-a_\mu(\square)}t^{l_\lambda(\square)+1} 
+
\sum_{\square\in\mu}
q^{a_\lambda(\square)+1}t^{-l_\mu(\square)}.
\]
Observe that
\begin{equation*}
N_{\lambda,\mu}(u)=\Omega\left[-u E_{\lambda,\mu}^{(0)} \right].
\label{NekOmega}
\end{equation*}
On the other hand, in \cite[Lemma 2.1]{MellitPoin1} (cf.~\cite[Lemma 6]{COExt}), it was shown that
\begin{align*}
E_{\lambda,\mu}&=qt B_{\lambda}+\overline{B}_{\mu}-(q-1)(t-1)B_\lambda \overline{B}_{\mu}\\
&= \frac{-qt}{(1-q)(1-t)}\big(D_\lambda\overline{D}_\mu-1 \big).\qedhere
\end{align*}
\end{proof}

\section{Wreath Macdonald theory}\label{WMT}
\subsection{Multi-symmetric functions}
Let $\Lambda$ be the ring of symmetric functions and further set $\Lambda_{q,t}\coloneq\Lambda\otimes\CC(q,t)$.
We have the usual distinguished elements of $\Lambda$ (cf.~\cite{Mac}):
\begin{itemize}
\item the \textit{power sum} $p_n$;
\item the \textit{elementary symmetric function} $e_n$;
\item the \textit{complete symmetric function} $h_n$;
\item the \textit{Schur function} $s_\lambda$ associated to a partition $\lambda$.
\end{itemize}
We will use plethystic notation, viewing them as functions in an alphabet of variables $X$: $p_n[X]$, $e_n[X]$, etc.

For our choice of $r$, we consider the rings $\Lambda^{\otimes r}$ and $\Lambda_{q,t}^{\otimes r}$.
The tensorands will be indexed from $0$ to $r-1$, which we view as elements of $\ZZ/r\ZZ$.
To the tensorand indexed by $i\in\ZZ/r\ZZ$, we assign an alphabet $X^{(i)}$; we call $i$ the color of the alphabet $X^{(i)}$.
Given $f\in\Lambda$, we denote by $f[X^{(i)}]\in\Lambda^{\otimes r}$ to be the element that is $f$ in the tensorand indexed by $i$ and 1 elsewhere.
For example, we have colored power sums $p_n[X^{(i)}]$, which, varying over $n$ and $i$, are polynomial generators of $\Lambda^{\otimes r}$.
A general element $f\in\Lambda^{\otimes r}$ can be a nontrivial function in every color of variable, so we denote it by $f[X^\bullet]$.
To an $r$-tuple of partitions $\vec{\lambda}=\left( \lambda^{(0)},\lambda^{(1)},\ldots,\lambda^{(r-1)} \right)$, we can associate the \textit{multi-Schur} function
\[
s_{\vec{\lambda}}[X^\bullet]\coloneq s_{\lambda^{(0)}}[X^{(0)}]s_{\lambda^{(1)}}[X^{(1)}]\cdots s_{\lambda^{(r-1)}}[X^{(r-1)}].
\]
Varying over all $r$-tuples of partitions, the multi-Schur functions give a basis of $\Lambda^{\otimes r}$.

\subsubsection{Matrix plethysm}
Because $\Lambda$ is a polynomial ring in the power sums $\{p_1,p_2,\ldots\}$, we can define ring homomorphisms out of $\Lambda$ by specifying the image of each $p_n$.
The usual plethysm, which we call here scalar plethysm, gives a convenient shorthand for certain homomorphisms of interest.
Given a series $E=E(u_1,u_2,\ldots)$ in some parameters $u_1,u_2, \ldots$, we set
\[
p_n[EX]= E(u_1^n,u_2^n,\ldots) p_n[X]
\]
and extend algebraically to define $f[EX]$ for any $f\in\Lambda$.
We can make sense of the case when $E$ is a rational function by taking series expansions.
For instance,
\[
p_n\left[ \frac{X}{1-q} \right]=\frac{p_n[X]}{1-q^n},
\] 
which gives a well-defined element of $\Lambda_{q,t}$.

For $\Lambda^{\otimes r}$, there is the richer structure of matrix plethysm.
Namely, we now take an $r\times r$ matrix (whose rows and columns are indexed by $\ZZ/r\ZZ$)
\[
M=\big( E_{i,j}(u_1,u_2,\ldots) \big)
\]
where each $E_{i,j}=E_{i,j}(u_1,u_2,\ldots)$ is a series.
The image of $p_n[X^{(i)}]$ is given by
\[
p_n[MX^{(i)}]\coloneq \sum_{j\in\ZZ/r\ZZ} p_n[E_{j,i}X^{(j)}],
\]
from which we define $f[MX^\bullet]$ for a general $f\in\Lambda^{\otimes r}$.
Two particular instances of this will be important for us: define $\sigma$ and $\iota$ by
\begin{align*}
p_n[\sigma X^{(i)}]&\coloneq p_n[X^{(i+1)}],\\
p_n[\iota X^{(i)}]&\coloneq p_n[X^{(-i)}].
\end{align*}
Setting $(-)^T$ as the transpose matrix, note that $\sigma^T=\sigma^{-1}$ and $\iota^T=\iota$.
The following can be checked by direct calculation.

\begin{lem}\label{PlethInv}
For any parameter $s$, the plethysm $(1-s\sigma^{\pm 1})$ is invertible with inverse given by
\[
p_n\left[ \frac{X^{(i)}}{(1-s\sigma^{\pm 1})} \right]=\frac{\sum_{j=0}^{r-1}s^{nj}p_n[X^{(i\pm j)}]}{1-s^{nr}}.
\]
\end{lem}

\subsubsection{Vector plethysm}
Next, we discuss evaluations.
In the scalar case, given a series $E=E(u_1,u_2,\ldots)$, we set
\[
p_n[E]:=E(u_1^n, u_2^n,\ldots).
\]
This is extended algebraically to define $f[E]$ for $f\in\Lambda$.
To evaluate elements of $\Lambda^{\otimes r}$, our inputs are subdivided along color:
\[
E^\bullet=\sum_{i\in\ZZ/r\ZZ} E^{(i)}.
\]
For $f\in\Lambda^{\otimes r}$, we define $f[E^\bullet]$ by mapping
\[
p_n[X^{(i)}]\mapsto p_n[E^{(i)}].
\]
Thus, one can view $E^\bullet$ as a vector plethysm.
The main instance of this kind of subdivision is the way we split characters according to their colors in \ref{CharSum}.
Finally, when we combine matrix and vector plethysms, we mean apply the matrix plethysm on the function first and then perform the vector plethysm evaluation.
For example, $f[\sigma^k\iota D_\lambda^\bullet]$ means we send
\[
p_n[X^{(i)}]\mapsto p_n[X^{(k-i)}]\mapsto p_n[D_\lambda^{(k-i)}].
\]
 
\subsubsection{$\Omega$ and $\TT$}\label{OT}
We define
\begin{equation}
\Omega[X]=\exp\left( \sum_{k>0}\frac{p_k[X]}{k} \right)=\sum_{n\ge 0}h_n[X].
\label{OmegaDef}
\end{equation}
Although this is an infinite sum of elements in $\Lambda$, each summand of fixed degree is finite.
Alternatively, it may be viewed as an element of the completion of $\Lambda$ by degree.
In $\Lambda^{\otimes r}$, we have $\Omega[X^{(i)}]$, and for any matrix plethysm $M$, we can make sense of 
\begin{align*}
\Omega[MX^{(i)}]&=\Omega\left[ \sum_{j\in\ZZ/r\ZZ}E_{j,i}X^{(j)}\right]\\
&=\Omega\left[ E_{0,i}X^{(0)}\right]\Omega\left[ E_{1,i}X^{(1)} \right]\cdots\Omega\left[ E_{r-1,i}X^{(i)} \right] .
\end{align*}

The translation operator $\TT[X]$ on $\Lambda$ is the ring automorphism defined by
\[
\TT[X](p_n[X])=p_n[X]+1.
\]
If we define the skewing operator by
\[
p_m^\perp[X]=m\frac{\partial}{\partial p_m[X]},
\]
then we have
\[
\TT[X]=\exp\left( \sum_{k>0} \frac{p_k^\perp[X]}{k} \right),
\]
in analogy with \eqref{OmegaDef}.
For a scalar plethysm $E=E(u_1,u_2,\ldots)$, we set
\begin{align*}
p_k^\perp[EX]&\coloneq E(u_1^k,u_2^k,\ldots) p_k^{\perp}[X]\\
\TT[EX]&\coloneq \exp\left( \sum_{k>0}\frac{p_k^\perp[EX]}{k} \right).
\end{align*}
Note then that
\[
\TT[EX]\left( p_n[X] \right)=p_n[X]+E(u_1^n,u_2^n,\ldots). 
\]
Finally, on $\Lambda^{\otimes r}$, we naturally define
\begin{align*}
p_m^\perp[X^{(i)}]&=m\frac{\partial}{\partial p_m[X^{(i)}]}\\
\TT[X^{(i)}]&=\exp\left( \sum_{k>0}\frac{p_k^\perp[X^{(i)}]}{k} \right).
\end{align*}
For any matrix plethysm $M$, we set
\begin{align*}
\TT[MX^{(i)}]&= \TT\left[ \sum_{j\in\ZZ/r\ZZ}E_{j,i}X^{(j)} \right]\\
&= \TT\left[ E_{0,i}X^{(0)} \right]\TT\left[ E_{1,i}X^{(1)} \right]\cdots\TT\left[ E_{r-1,i}X^{(r-1)} \right].\\
\end{align*}

\begin{lem}[\protect{\cite[Lemma 3.9]{RW}}]\label{OTComm}
For a pair of matrix plethysms $A$ and $B$, we have
\[
\TT\left[ AX^{(j)}z \right]\Omega\left[ BX^{(k)}w \right]
=
\Omega\left[ (A^TB)_{j,k}zw \right]\Omega\left[ BX^{(k)}w \right]\TT\left[ AX^{(j)}z \right].
\]
\end{lem}

\subsection{Wreath Macdonald polynomials}
The following is equivalent to the definition given by Haiman \cite{Haiman}.

\begin{defn}
Given a partition $\lambda$, the \textit{wreath Macdonald polynomial} $H_\lambda[X^\bullet;q,t]\in\Lambda^{\otimes r}_{q,t}$ is determined by three conditions:
\begin{itemize}
\item $H_\lambda\left[ (1-q\sigma^{-1})X^\bullet;q,t \right]\in\mathrm{span}\left\{ s_{\quot(\mu)}\,\middle|\, \mu\ge\lambda\hbox{ and } \core(\mu)=\core(\lambda) \right\}$;
\item $H_\lambda\left[ (1-t^{-1}\sigma^{-1})X^\bullet; q,t \right]\in\mathrm{span}\left\{ s_{\quot(\mu)}\,\middle|\, \mu\le\lambda\hbox{ and } \core(\mu)=\core(\lambda) \right\}$;
\item $H_\lambda[1]=1$.
\end{itemize}
When we have no need to emphasize the $(q,t)$ dependence, we may simply denote it by $H_\lambda[X^\bullet]$ or even $H_\lambda$.
Note that $\deg(H_\lambda)=|\quot(\lambda)|$.
\end{defn}
Bezrukavnikov--Finkelberg \cite{BezFink} first proved the existence of $H_\lambda$ as well as a version of multi-Schur positivity.
An alternative proof of existence was given in \cite{WreathEigen}.

Bases of $\Lambda^{\otimes r}$ are naturally indexed by $r$-tuples of partitions.
If we fix an $r$-core $\alpha$, then $\{H_\lambda\,|\, \core(\lambda)=\alpha\}$ does indeed form a basis of $\Lambda_{q,t}^{\otimes r}$.
One can view the $r$-core as imposing an order on $r$-tuples of partitions by going backwards along the core-quotient decomposition and applying the usual dominance order.

\subsubsection{Pairing}
Recall that $\Lambda$ possesses the Hall inner product $\langle-,-\rangle$, for which the Schur functions form an orthonormal basis.
We use the same notation for the tensor product pairing on $\Lambda^{\otimes r}$; now the multi-Schur functions form an orthonormal basis:
\[
\langle s_{\vec{\lambda}},s_{\vec{\mu}}\rangle=\delta_{\vec{\lambda},\vec{\mu}}.
\]
The (modified) \textit{wreath Macdonald pairing} $\langle -, - \rangle_{q,t}'$ on $\Lambda_{q,t}^{\otimes r}$ is defined to be
\[
\left\langle f, g\right\rangle_{q,t}'\coloneq\left\langle f\left[\iota X^\bullet\right], g\left[(1-q\sigma^{-1})(t\sigma-1)X^\bullet  \right]\right\rangle.
\]
This pairing is symmetric.

\begin{lem}[\protect{\cite[Lemma 3.10]{RW}}]\label{AdjointLem}
We have the following adjunction relation:
\[
\left\langle \Omega[X^{(i)}]f, g\right\rangle_{q,t}'
=
\left\langle f, \TT\left[ (1-q\sigma)(t\sigma^{-1}-1)X^{(-i)} \right]g\right\rangle_{q,t}'.
\]
\end{lem}

\subsubsection{Dagger polynomials}
For $r\ge 3$, $\{H_\lambda\}$ are no longer orthogonal with respect to $\langle -,-\rangle_{q,t}'$.
A dual orthogonal basis is given by the ``dagger'' polynomials:
\[
H_\lambda^\dagger[X^\bullet;q,t]\coloneq H_\lambda[-\iota X^\bullet; q^{-1},t^{-1}].
\]

\begin{prop}[\protect{\cite[Proposition 2.12]{RW}}]\label{DaggerProp}
For $\lambda$ and $\mu$ with $\core(\lambda)=\core(\mu)$, $\langle H_\lambda^\dagger,H_\mu\rangle_{q,t}'$ is nonzero if and only if $\lambda=\mu$.
\end{prop}

The inner product $\langle H_\lambda^\dagger, H_\lambda\rangle_{q,t}'$ was computed in \cite{OSWreath}.
We will give an independent derivation in this paper.

\subsubsection{Nabla operator}
We will define a nabla operator $\nabla_\alpha$ for each $r$-core $\alpha$.
For $\lambda$ with $\core(\lambda)=\alpha$, we set 
\[
\nabla_\alpha H_\lambda= \left( \prod_{\substack{\square\in\lambda\backslash\alpha\\ \bar{c}_\square=0}}-\chi_\square \right)H_\lambda.
\]
Let $\nabla^\dagger_\alpha$ denote the adjoint of $\nabla$ with respect to $\langle -, -\rangle_{q,t}'$.
A corollary of Proposition \ref{DaggerProp} is that
\[
\nabla_\alpha^{\dagger} H_\lambda^\dagger= \left( \prod_{\substack{\square\in\lambda\backslash\alpha\\ \bar{c}_\square=0}}-\chi_\square \right)H_\lambda^\dagger.
\]

\begin{rem}
In \cite{RW}, it was shown that $\nabla_\alpha^\dagger=\nabla_{w_0\alpha}$, where $w_0\alpha$ is defined using an action of the symmetric group $\Sigma_r$ on partitions.
This characterization will not be important in this paper, but we mention this to facilitate comparing results from \cite{RW} to how we use them here.
\end{rem}

\section{Ext operator}\label{ExtOp}

\subsection{Tesler identity}
One can frame the work of Carlsson, Nekrasov and Okounkov \cite{CNO} in the $r=1$ case in terms of the \textit{Tesler identity} \cite{GHT}.
This was done implicitly in \cite{MellInt}.
We outline this approach for the wreath case.

\subsubsection{Delta functions}
Define the series
\begin{align*}
\mathbb{E}_\lambda&\coloneq \Omega\left[ \sum_{i\in\ZZ/r\ZZ}X^{(i)}\left( \frac{D_\lambda}{(1-q)(t-1)} \right)^{(i)} \right]\\
\mathbb{E}_\lambda^*&\coloneq \Omega\left[ \sum_{i\in\ZZ/r\ZZ} X^{(i)}\left( \frac{-\overline{D}_\lambda}{(1-q^{-1})(t^{-1}-1)} \right)^{(i)} \right].
\end{align*}
They are delta functions with respect to the pairing $\langle -, -\rangle_{q,t}'$.
\begin{prop}\label{DeltaProp}
For any $f\in\Lambda_{q,t}^{\otimes r}$, we have
\begin{align}
\label{Delta}
\langle f, \mathbb{E}_\lambda\rangle_{q,t}'&= f[\iota D_\lambda^\bullet],\\
\label{DeltaStar}
\langle f, \mathbb{E}_\lambda^*\rangle_{q,t}'&= f[-qt\iota \overline{D}_\lambda^\bullet].
\end{align}
\end{prop}

\begin{proof}
Equation \eqref{Delta} is \cite[Corollary 4.10]{RW} and equation \eqref{DeltaStar} is the $k=0$ case of \cite[Proposition 5.12]{RW}.
For the latter, we recall that in \ref{CharSum}, our $\overline{(-)}$ operation assigns color differently than the $(-)_*$ operation from \cite{RW}. 
\end{proof}

\subsubsection{The maps $\mathsf{V}$ and $\mathsf{V}^*$}
The following give wreath analogues of the universal sheaf operators from \cite{CNO}:
for each $r$-core $\alpha$, we define
\begin{align*}
\mathsf{V}_\alpha&\coloneq \nabla_\alpha\Omega\left[ \frac{X^{(0)}}{(1-q\sigma^{-1})(t\sigma-1)} \right]\TT\left[X^{(0)}\right]\\
\mathsf{V}_\alpha^*&\coloneq \left(\nabla_\alpha^\dagger\right)^{-1}\Omega\left[ \frac{-qtX^{(0)}}{(1-q\sigma^{-1})(t\sigma-1)} \right]\TT\left[ -X^{(0)} \right].
\end{align*}
The following is a slight rewriting of Theorem 1.1 and Proposition 5.14 of \cite{RW}.
\begin{thm}\label{Tesler}
For $\lambda$ with $\core(\lambda)=\alpha$, we have
\begin{align*}
\mathsf{V}_\alpha H_\lambda &= \mathbb{E}_\lambda,\\
\mathsf{V}_\alpha^*H_\lambda^\dagger &= \mathbb{E}_\lambda^*.
\end{align*}
\end{thm}

\subsection{Vertex operator}
Our goal here is to write the Nekrasov factor $N_{\lambda,\mu}(u)$ using a vertex operator.

\subsubsection{Nekrasov factor via $\mathsf{V}_\alpha$ and $\mathsf{V}_\alpha^*$}
Let $u^D$ denote the grading operator on $\Lambda^{\otimes r}$: for a homogeneous $f\in\Lambda^{\otimes r}$, we set
\[
u^D f= u^{\deg (f)}f.
\]
It clearly commutes with $\nabla_\alpha$.

\begin{lem}
For $\lambda$ and $\mu$ with $\core(\lambda)=\core(\mu)=\alpha$, we have
\begin{equation}
\left\langle \mathsf{V}_\alpha^*H_\mu^\dagger, u^D\mathsf{V}_\alpha H_\lambda\right\rangle_{q,t}'
=
\Omega\left[ \left(\frac{uqt}{(1-q)(1-t)}\right)^{(0)} \right]N_{\lambda,\mu}(u).
\label{VNek}
\end{equation}
\end{lem}

\begin{proof}
Applying Theorem \ref{Tesler} and then \eqref{DeltaStar}, we have 
\begin{align*}
\left\langle \mathsf{V}_\alpha^*H_\mu^\dagger, u^D\mathsf{V}_\alpha H_\lambda\right\rangle_{q,t}'
&= \left\langle\mathbb{E}_\mu^*, \Omega\left[ \sum_{i\in\ZZ/r\ZZ}uX^{(i)}\left( \frac{D_\lambda}{(1-q)(t-1)} \right)^{(i)} \right]\right\rangle_{q,t}'\\
&= \Omega\left[ \sum_{i\in\ZZ/r\ZZ}uqt\overline{D}_{\mu}^{(-i)}\left( \frac{D_\lambda}{(1-q)(1-t)} \right)^{(i)} \right].
\end{align*}
Observe now that
\[
\sum_{i\in\ZZ/r\ZZ}uqt\overline{D}_{\mu}^{(-i)}\left( \frac{D_\lambda}{(1-q)(1-t)} \right)^{(i)}
=\left(\frac{uqt D_\lambda\overline{D}_\mu}{(1-q)(1-t)}\right)^{(0)}.
\]
By Lemma \ref{NekExp}, we obtain the right-hand-side of \eqref{VNek} upon applying the plethystic exponential.
\end{proof}

\subsubsection{The operator $\mathsf{W}(u)$} 
Let $\left( \mathsf{V}_{\alpha}^* \right)^\dagger$ denote the adjoint of $\mathsf{V}_\alpha^*$ with respect to $\langle -, -\rangle_{q,t}'$.
Thus, we can reframe (\ref{VNek}) in terms of $\left( \mathsf{V}_\alpha^* \right)^\dagger u^D\mathsf{V}_\alpha$.

\begin{lem}
We have the equality
\begin{equation}
\left( \mathsf{V}_\alpha^* \right)^\dagger u^D\mathsf{V}_\alpha
=
\Omega\left[\left(\frac{uqt}{(1-q)(1-t)}\right)^{(0)}  \right]\Omega\left[ \frac{(u-1)X^{(0)}}{(1-q\sigma^{-1})(t\sigma-1)}\right]
\TT\left[ (u^{-1}-qt)X^{(0)} \right]u^D
\label{VStarV}
\end{equation}
\end{lem}

\begin{proof}
First, we use Lemma \ref{AdjointLem} to compute
\[
\left( \mathsf{V}_\alpha^* \right)^\dagger=\Omega\left[ \frac{-X^{(0)}}{(1-q\sigma^{-1})(t\sigma-1)} \right]\TT\left[ -qtX^{(0)} \right]\nabla_\alpha^{-1}.
\]
Thus, we have
\begin{align*}
\left( \mathsf{V}_\alpha^* \right)^\dagger u^D\mathsf{V}_\alpha
&= \Omega\left[ \frac{-X^{(0)}}{(1-q\sigma^{-1})(t\sigma-1)} \right]\TT\left[ -qtX^{(0)} \right]\nabla_\alpha^{-1}
u^D
\nabla_\alpha
\Omega\left[ \frac{X^{(0)}}{(1-q\sigma^{-1})(t\sigma-1)} \right]\TT\left[X^{(0)}\right]\\
&= 
\Omega\left[ \frac{-X^{(0)}}{(1-q\sigma^{-1})(t\sigma-1)} \right]
\underbrace{\TT\left[ -qtX^{(0)} \right]
\Omega\left[ \frac{uX^{(0)}}{(1-q\sigma^{-1})(t\sigma-1)} \right]}
\TT\left[uX^{(0)}\right]
u^D.
\end{align*}
Using Lemma \ref{OTComm} to commute the braced terms, we obtain a a factor of
\[
\Omega\left[ \left(\frac{uqt}{(1-q\sigma^{-1})(1-t\sigma)}\right)_{0,0} \right].
\]
Applying Lemma \ref{PlethInv}, we have
\begin{align*}
uqt\left(\frac{1}{(1-q\sigma^{-1})(1-t\sigma)}\right)_{0,0}
&= uqt\left(\frac{\sum_{a,b=0}^{r-1}q^at^b\sigma^{-a+b}}{(1-q^r)(1-t^r)}\right)_{0,0}\\
&=\frac{uqt\sum_{a=0}^{r-1}(qt)^a}{(1-q^r)(1-t^r)}\\
&= \left( \frac{uqt}{(1-q)(1-t)} \right)^{(0)}.\qedhere
\end{align*}
\end{proof}

\begin{defn}
We define the Ext operator
\begin{align}
\mathsf{W}(u)
&\coloneq\Omega\left[ \frac{(1-u^{-1})X^{(0)}}{(1-q\sigma^{-1})(t\sigma-1)}\right]
\TT\left[ (1-uqt)X^{(0)} \right]
\label{WDef}\\
&\hphantom{:}=u^{-D}\left(\Omega\left[ \frac{(u-1)X^{(0)}}{(1-q\sigma^{-1})(t\sigma-1)}\right]
\TT\left[ (u^{-1}-qt)X^{(0)} \right]u^D  \right).
\nonumber
\end{align}
\end{defn}
Splicing together \eqref{VStarV} and \eqref{VNek}, we obtain the following corollary.
\begin{cor}\label{WNek}
$\mathsf{W}(u)$ reproduces the Nekrasov factor when $\core(\lambda)=\core(\mu)$:
\begin{equation}
\left\langle H_\mu^\dagger, \mathsf{W}(u)H_\lambda\right\rangle_{q,t}'=u^{-|\quot(\mu)|}N_{\lambda,\mu}(u).
\label{ExtPair}
\end{equation}
\end{cor}

Notice that we have no problem setting $u=1$ in the definition of $\mathsf{W}(u)$ and the Nekrasov factor in \eqref{ExtPair}.

\begin{cor}[\protect{\cite[Theorem~3.32]{OSWreath}}]\label{Norm}
We have 
\[\langle H_\lambda^\dagger, H_\lambda\rangle_{q,t}'=N_{\lambda,\lambda}(1)
=
\prod_{\substack{\square\in\lambda\\ h_{\lambda}(\square)\equiv 0\:\mathrm{mod}\: r}}
\left( 1-q^{-a_\lambda(\square)}t^{l_\lambda(\square)+1} \right)
\left( 1-q^{a_\lambda(\square)+1}t^{-l_\lambda(\square)} \right).
\]
\end{cor}

\begin{proof}
Setting $u=1$ and $\lambda=\mu$ in \eqref{ExtPair}, we obtain
\[
\left\langle H_\lambda^\dagger, \TT\left[ (1-qt)X^{(0)} \right]H_\lambda\right\rangle_{q,t}'=N_{\lambda,\lambda}(1).
\]
The summand of highest degree in $\TT\left[ (1-qt)X^{(0)} \right]H_\lambda$ is $H_\lambda$.
\end{proof}

\begin{rem}
One can come up with operators that produce \eqref{ExtPair} when $\core(\lambda)\not=\core(\mu)$.
Geometrically, this corresponds to the Ext correspondence over a pair of moduli of sheaves on the $A_{r-1}$ surface with different stability conditions.
However, the operator is not as nice because $(\nabla_{\core(\mu)}^\dagger)^{-1}$ and $\nabla_{\core(\lambda)}$ no longer cancel.
\end{rem}

\subsection{Pieri rules}
Following the idea in the proof of Corollary \ref{Norm}, we can specialize $u$ to express certain Pieri rules in terms of Nekrasov factors.
We note that these are not wreath analogues of the Pieri rules from \cite[Chapter VI]{Mac}.
It is still an open problem to find manicured formulas for the latter \cite{OSW, WreathOrth}.

\subsubsection{Support}
Let $\mu\subset_k\lambda$ denote that:
\begin{itemize}
\item $\mu\subset\lambda$ and
\item $\lambda\backslash\mu$ contains $k$ boxes of each color.
\end{itemize}
It follows that $\core(\mu)=\core(\lambda)$.

For $F\in\Lambda^{\otimes r}$, let $F^\perp$ denote the adjoint of multiplication by $F$ with respect to the Hall pairing $\langle -, - \rangle$.
Thus, if we write $F$ in terms of $\{p_n[X^{(i)}]\}$, $F^\perp$ is obtained by replace each $p_n[X^{(i)}]$ by $p_n^\perp[X^{(i)}]$.

\begin{lem}[\protect{\cite[Lemma 4.3]{RW}}]\label{SupportLem}
For $F\in\Lambda^{\otimes r}$ of degree $k$, we have
\begin{align*}
FH_\lambda&= \sum_{\lambda\subset_k\mu}c_{\mu\lambda}H_\mu\\
F^\perp H_\lambda&= \sum_{\mu\subset_k\lambda} c_{\mu\lambda}^*H_\mu.
\end{align*}
\end{lem}

\subsubsection{Skewing}
Let us define $h_n^\perp[X]\coloneq \left(h_n[X]\right)^\perp$.
Thus, we have
\[
\sum_{n=0}^\infty h_n^\perp [X]z^n=\TT[Xz].
\]
We define plethysm for $h_n^\perp[X]$ by transforming $p_k^\perp[X]$ as in \ref{OT}.

\begin{prop}\label{SPieri}
We have
\[
h_n^\perp\left[ (1-qt)X^{(0)} \right]H_\lambda= \sum_{\mu\subset_{n}\lambda}\frac{N_{\lambda,\mu}(1)}{N_{\mu,\mu}(1)}H_\mu.
\]
\end{prop}

\begin{proof}
As in the proof of Corollary \ref{Norm}, we set $u=1$ in Corollary \ref{WNek}.
The partitions $\mu$ that appear are controlled by Lemma \ref{SupportLem}.
\end{proof}

\subsubsection{Multiplication}
When we set $u=(qt)^{-1}$ in (\ref{WDef}), we have
\[
\mathsf{W}(q^{-1}t^{-1})=\Omega\left[ \frac{(1-qt)X^{(0)}}{(1-q\sigma^{-1})(t\sigma-1)} \right].
\]
Combining this with Corollaries \ref{WNek} and \ref{Norm} and Lemma \ref{SupportLem} gives us the following.
\begin{prop}\label{MPieri}
We have
\[
e_n\left[ \frac{(1-qt)X^{(0)}}{(1-q\sigma^{-1})(1-t\sigma)} \right]H_\lambda=
\sum_{\lambda\subset_n\mu} (-1)^n(qt)^{|\quot(\mu)|}\frac{N_{\lambda,\mu}(q^{-1}t^{-1})}{N_{\mu,\mu}(1)}H_\mu.
\]
\end{prop}

\begin{rem}
Let us make two comments about both Propositions \ref{SPieri} and \ref{MPieri}:
\begin{enumerate}
\item At $n=1$, one obtains the Pieri rules (at color 0) computed in the appendix to \cite{RW5Term}.
A direct combinatorial comparison with \textit{loc.~cit.}\ would be an interesting exercise.
\item We used Lemma \ref{SupportLem} to constrain the partitions appearing in the outputs of both Pieri rules.
However, it would be interesting to conclude this without the lemma by looking at the zeros of $N_{\lambda,\mu}(1)$ and $N_{\lambda,\mu}(q^{-1}t^{-1})$.
\end{enumerate}
\end{rem}

\subsection{Nekrasov--Okounkov formulas}
We now set out to prove Theorem~\ref{Thm_WW} by taking the graded trace of the vertex operator $\mathsf{W}(u)$.

\subsubsection{Graded trace}
In this subsection, we will produce interesting sum-to-product formulas by taking graded traces of certain operators on $\Lambda_{q,t}^{\otimes r}$.

\begin{lem}
Let $T$ be a variable such that $|T|\ll 1$; we assume that the magnitude of $T$ is less than all other variables involved.
For matrix plethysms $A$ and $B$, we have
\begin{equation}
\mathrm{Tr}_{\Lambda_{q,t}^{\otimes r}}\left( \Omega\left[ BX^\bullet \right]\TT\left[ AX^\bullet \right]T^{D} \right)
=\frac{1}{(T;T)^r_\infty}\Omega\left[\frac{T}{1-T} \sum_{i,j=0}^{r-1}(A^TB)_{i,j} \right]
\label{WtTrace}
\end{equation}
\end{lem}

\begin{proof}
This result is standard, but we give a proof nonetheless.
Let us introduce a parameter $z$ and consider
\[
\mathrm{Tr}(z)\coloneq\mathrm{Tr}_{\Lambda_{q,t}^{\otimes r}}\left( \Omega\left[BX^\bullet \right]\TT\left[ zAX^\bullet \right]T^{D} \right).
\]
By the cyclicity of the trace and Lemma \ref{OTComm}, we have
\begin{align*}
\mathrm{Tr}(z)&=\mathrm{Tr}_{\Lambda_{q,t}^{\otimes r}}\left(\TT\left[TzAX^\bullet \right] \Omega\left[BX^\bullet \right]T^{D} \right)\\
&= \Omega\left[ Tz\sum_{i,j=0}^{r-1}(A^TB)_{i,j} \right]\mathrm{Tr}(Tz).
\end{align*}
This gives a recursion.
Setting $z=0$, we have
\[
\mathrm{Tr}(0)=\frac{1}{\left( T;T \right)^r_\infty}.
\]
Solving the recursion gives us
\[
\mathrm{Tr}(z)=\frac{1}{\left( T;T \right)^r_\infty}
\Omega\left[\frac{Tz}{1-T} \sum_{i,j=0}^{r-1}(A^TB)_{i,j} \right]
\]
and we obtain \eqref{WtTrace} by setting $z=1$. 
\end{proof}

\subsubsection{The Walsh--Warnaar conjecture}
With the previous lemma under our belts we are ready to prove the Walsh--Warnaar conjecture, our Theorem~\ref{Thm_WW}.

\begin{proof}[Proof of Theorem~\ref{Thm_WW}]
We will prove the theorem with $t$ inverted.
First, note that the left-hand side of \eqref{WW1} is $\mathrm{Tr}_{\Lambda_{q,t}^{\otimes r}}\left(\mathsf{W}(u)T^D\right)$ with respect to the basis $\{H_\lambda\,|\,\core(\lambda)=\alpha\}$.
To see that, we combine Corollaries \ref{WNek} and \ref{Norm} to obtain
\[
\mathsf{W}(u)T^DH_\lambda=\sum_{\substack{\mu\\\core(\mu)=\alpha}}T^{|\quot(\lambda)|}u^{-|\quot(\mu)|}\frac{N_{\lambda,\mu}(u)}{N_{\mu,\mu}(1)}H_\mu.
\]
At $\mu=\lambda$, we can use Proposition \ref{QuotHook} to write the matrix element as in \eqref{WW1} with $t$ inverted.

Applying \eqref{WtTrace}, we obtain
\begin{align}
\nonumber
\mathrm{Tr}_{\Lambda_{q,t}^{\otimes r}}\left(\mathsf{W}(u)T^D\right)
&=
\frac{1}{\left( T;T \right)^r_\infty}\Omega\left[-\frac{T(1-u^{-1})(1-uqt)}{1-T}\left( \frac{1}{(1-q\sigma^{-1})(1-t\sigma)} \right)_{0,0} \right]\\
\label{WW3}
&=\frac{1}{\left( T;T \right)^r_\infty}\Omega\left[-\frac{T(1+qt-u^{-1}-uqt)}{(1-T)(1-q^r)(1-t^r)}\sum_{i=0}^{r-1} q^it^i\right].
\end{align}
In the second equality, we note that we have already computed the matrix element in the proof of Lemma \ref{VNek}.
Assuming $|T|$ to be very small, we can perform the following manipulations:
\begin{align*}
&\Omega\left[-\frac{T(1+qt-u^{-1}-uqt)}{(1-T)(1-q^r)(1-t^r)}\sum_{i=0}^{r-1} q^it^i\right]\\
&=
\Omega\left[\frac{T(1+qt-u^{-1}-uqt)}{(1-T)(1-q^r)(1-t^{-r})}\sum_{i=0}^{r-1} q^it^{-r+i}\right]\\
&=
\Omega\left[\frac{T}{(1-T)(1-q^r)(1-t^{-r})}\left((1-u)\sum_{i=1}^{r} q^it^{-r+i}+(1-u^{-1})\sum_{i=1}^{r} q^{r-i}t^{-i}\right)\right]\\
&= 
\prod_{i=1}^r\frac{\left(uq^it^{-r+i}T, u^{-1}q^{r-i}t^{-i}T; q^r, t^{-r}, T  \right)_\infty}{\left( q^it^{-r+i}T, q^{r-i}t^{-i}T; q^r,t^{-r}, T \right)_\infty}.
\end{align*}
Plugging this into \eqref{WW3}, we obtain \eqref{WW1} with $t$ inverted.
\end{proof}

\subsubsection{Elliptic Nekrasov--Okounkov formulas}
Walsh and Warnaar make a further conjecture regarding an elliptic analogue of the modular $(q,t)$-Nekrasov--Okounkov formula.
For completeness, we discuss this here.

As pointed out by Rains and Warnaar \cite[Appendix~A]{RainsWar16}, the equivariant Dijkgraaf--Moore--Verlinde--Verlinde (DMVV) formula, originally conjectured by Li, Liu and Zhou \cite{LLZ06} and proved by Waelder \cite{Waelder08}, can be viewed as an elliptic analogue of the $(q,t)$-Nekrasov--Okounkov formula \eqref{Eq_qtNO}.
To state this, let $\Ell(\Hilb_n(\mathbb{C}^2);u,p,t_1,t_2)$ denote the equivariant elliptic genus of the Hilbert scheme, where for our purposes $(u,p,t_1,t_2)$ are regarded as formal variables.
With the aid of the modified theta function $\theta(z;p)\coloneq(z,p/z;p)_\infty$, for which we adopt the multiplicative notation
\[
\theta(z_1,\dots,z_k;p)\coloneq \theta(z_1;p)\cdots\theta(z_k;p),
\]
the equivariant elliptic genus of $\mathbb{C}^2$ is given by
\begin{align*}
\mathrm{Ell}(\mathbb{C}^2;u,p,t_1,t_2)
&=\frac{\theta(ut_1^{-1},u^{-1}t_2;p)_\infty}{\theta(t_1^{-1},t_2;p)}
=\sum_{m\ge 0}\sum_{\ell,n_1,n_2\in\mathbb{Z}}c(m,\ell,n_1,n_2)p^m u^\ell t_1^{n_1}t_2^{n_2}.
\end{align*}
The symmetry in $(t_1,t_2)$ can be seen from the symmetry of the theta function: $\theta(z;p)=-z\theta(1/z;p)$.
Using the integers $c(m,k,n_1,n_2)$ the equivariant DMVV formula then states that
\begin{equation}\label{EqDMVV}
\sum_{n\ge 0}T^n\Ell\!\big(\Hilb_n(\mathbb{C}^2);u,p,t_1,t_2\big)
=\prod_{m\ge 0}\prod_{k\ge 1}\prod_{\ell,n_1,n_2\in\mathbb{Z}}
\frac{1}{(1-p^mT^ku^\ell t_1^{n_1}t_2^{n_2})^{c(km,\ell,n_1,n_2)}}.
\end{equation}
By localization the generating function for elliptic genera of the Hilbert scheme may be alternatively expressed as
\[
\sum_{n\ge 0}T^n\Ell\!\big(\Hilb_n(\mathbb{C}^2);u,p,t_1,t_2\big)
=\sum_{\lambda}T^{|\lambda|}
\prod_{\square\in\lambda}\frac{\theta(ut_1^{-a_\lambda(\square)-1}t_2^{l_\lambda(\square)},u^{-1}t_1^{-a_\lambda(\square)}t_2^{l_\lambda(\square)+1};p)}
{\theta(t_1^{-a_\lambda(\square)-1}t_2^{l_\lambda(\square)},t_1^{-a_\lambda(\square)}t_2^{l_\lambda(\square)+1};p)}.
\]
Thus, \eqref{EqDMVV} can be viewed as an elliptic lift of \eqref{Eq_qtNO}.
Indeed, the content of \cite[Appendix~A]{RainsWar16} is to show that an appropriate rewriting of the product side of \eqref{EqDMVV} gives an expression which manifestly reduces to the $(q,t)$-Nekrasov--Okounkov formula in the $p\to 0$ limit (and upon setting $(t_1,t_2)\mapsto(q^{-1},t)$).

Motivated by this and their modular conjecture in the $(q,t)$-case, Walsh and Warnaar further conjectured that for each $r$-core $\alpha$ the sum
\begin{equation}\label{WWtheta}
\sum_{\substack{\lambda\\\core(\lambda)=\alpha}}T^{|\quot(\lambda)|}
\prod_{\substack{\square\in\lambda\\h_\lambda(\square)\equiv0\;\mathrm{mod}\;r}}\frac{\theta(ut_1^{-a_\lambda(\square)-1}t_2^{l_\lambda(\square)},u^{-1}t_1^{-a_\lambda(\square)}t_2^{l_\lambda(\square)+1};p)}
{\theta(t_1^{-a_\lambda(\square)-1}t_2^{l_\lambda(\square)},t_1^{-a_\lambda(\square)}t_2^{l_\lambda(\square)+1};p)}
\end{equation}
is independent of $\alpha$.
They were able to prove this conjecture for the coefficient of $T^n$ with $0\le n\le 2$, the last case requiring heavy use of addition formulas for theta functions.
Replacing $(t_1,t_2)\mapsto(q^{-1},t)$ and taking the $p\to0$ limit recovers the sum-side of Theorem~\ref{Thm_WW}. Thus, our result provides further evidence for this conjecture.
It is also an open problem to determine a ratio of modified theta functions which governs the expansion of the sum \eqref{WWtheta} as an infinite product, as in \eqref{EqDMVV}.

\subsection*{Acknowledgments}
We would like to thank
Noah Arbesfeld,
Anton Mellit,
Marino Romero and 
Ole Warnaar for helpful conversations.
S.~Albion Ferlinc was partially supported by the Austrian Science Fund (FWF) 10.55776/F1002 in the framework of the Special Research Program ``Discrete Random Structures: Enumeration and Scaling Limits''.
J.~J.~Wen was supported by ERC consolidator grant No.~101001159 ``Refined invariants in combinatorics, low-dimensional topology and geometry of moduli spaces''.

\bibliographystyle{alpha}
\bibliography{Wreath}

\end{document}